\documentclass[11pt]{amsart}
\usepackage{amsfonts, amssymb, amsmath, amsthm, euscript, color}
\usepackage{geometry}
\usepackage{graphicx}
\usepackage{amssymb}
\usepackage{epstopdf}
\usepackage{url}
\usepackage[all]{xy}

\theoremstyle{plain}
\newtheorem{thm}{Theorem}[section]
\newtheorem{lem}[thm]{Lemma}

\newtheorem{cor}[thm]{Corollary}

\theoremstyle{definition}

\newtheorem*{thA}{Theorem A}
\newtheorem*{thB}{Theorem B}
\newtheorem*{thC}{Theorem C}
\newtheorem{deff}[thm]{Definition}
\newtheorem{rem}[thm]{Remark}

\newtheorem{remark}[thm]{Remark}

\theoremstyle{remark}

\newcommand*{\e}{\ensuremath{\varepsilon}}

\newcommand*{\Hom}{\ensuremath{\text{\upshape Hom}}}

\newcommand*{\HL}{\ensuremath{\text{\upshape H}}}

\newcommand*{\Ker}{\ensuremath{\text{\upshape Ker}}}

\newcommand*{\ad}{\ensuremath{\text{\upshape ad}}}

\newcommand*{\Img}{\ensuremath{\text{\upshape Im}}}
\newcommand*{\Id}{\ensuremath{\text{\upshape Id}}}

\newcommand*{\Chara}{\ensuremath{\text{\upshape char}}}

\def\dim{\operatorname{dim}}
\def\chara{\operatorname{char}}

\begin{document}
\thispagestyle{empty}

\title{A note on Masuoka's Theorem for semisimple irreducible Hopf algebras}

\author{Xingting Wang}
\email{xingting.wang@howard.edu}
\address{Department of Mathematics\\Howard University, Washington, D.C. 20059}

\thanks{The author was partially supported by U.~S.~National Science Foundation.}






\begin{abstract}
Masuoka proved (2009) that a finite-dimensional irreducible Hopf algebra $H$ in positive characteristic is semisimple if and only if it is commutative semisimple if and only if the Hopf subalgebra generated by all primitives is semisimple. In this note, we give another proof of this result by using Hochschild cohomology of coalgebras. 
\end{abstract}

\maketitle

\section{introduction}
The classification of semisimple Hopf algebras over an algebraically closed field is an open question \cite{andruskiewitsch2000finite}. By a result of Larson and Radford \cite{larson1988finite}, semisimple Hopf algebras are cosemisimple in characteristic zero. Nonetheless, there are plenty examples of semisimple Hopf algebras which are not cosemisimple in positive characteristic. For instance, suppose $p$ is a prime number and the base field $k$ has characteristic $p$. Then for any $p$-group $G$, $(kG)^*$ is semisimple but not cosemisipmle. Next, we recall an old theorem of Hochschild about restricted Lie algebras. Still assume $\Chara(k)=p$ and let $\mathfrak g$ be a restricted Lie algebra over $k$. Denote $u(\mathfrak g)$ as the restricted enveloping algebra of $\mathfrak g$. Hochschild \cite{hochschild1954representations} proves that $u(\mathfrak g)$ is semisimple if and only if $\mathfrak g$ is abelian and $\mathfrak g=k\mathfrak g^p$. From this, we can obtain the following assertion:

\begin{thA}[Hochschild]
Suppose $\Chara(k)=p$ and $k=\bar{k}$. Let $H$ be a finite-dimensional cocommutative connected Hopf algebra over $k$. Denote $\mathfrak g$ as the primitive space of $H$ and further assume that $H$ is primitively generated by $\mathfrak g$. Then the following are equivalent:
\begin{itemize}
\item[(1)] $H$ is semisimple.
\item[(2)] The restricted map ($p$-th power map) on $\mathfrak g$ is bijective.
\item[(3)] $H\cong (kG)^*$, for $G=(C_p)^n$ and $C_p$ the cyclic group of order $p$.
\end{itemize}
\end{thA}

In algebraic group theory, Nagata \cite{Na} proves that a fully reducible connected affine algebraic group in positive characteristic is a torus. Later, this result was generalized independently by Demazure-Gabriel \cite{demazure1970groupes} and Sweedler \cite{sweedler1971connected} to any connected fully reducible affine group schemes. Here we only recall the Hopf algebra version of the generalized Nagata's theorem. Suppose $\Chara(k)=p$ and $H$ is a commutative cosemisimple Hopf algebra over $k$, whose maximal separable subalgebra is $k$. Then $H$ is  cocommutative. In finite-dimensional case, we can dualize the result for semisimple cocommutative Hopf algebras.

\begin{thB}[Demazure-Gabriel and Sweedler]
Suppose $\Chara (k)=p$ and $k=\bar{k}$. Let $H$ be a finite-dimensional cocommutative connected Hopf algebra over $k$. Then the following are equivalent
\begin{itemize}
\item[(1)] $H$ is semisimple.
\item[(2)] $H\cong (kG)^*$, for $G$ an abelian $p$-group.
\end{itemize}
\end{thB}

As we can see Theorem B is a generalization of Theorem A without the assumption that the cocommutative Hopf algebra is primitively generated. Masuoka \cite{MA} further extends Theorem $B$ to include the case for all finite-dimensional connected Hopf algebras. Moreover he generalizes the result by providing the criteria when these Hopf algebras are semisimple.

\begin{thC}[Masuoka]
Suppose $\Chara(k)=p$ and $k=\bar{k}$. Let $H$ be a finite-dimensional connected Hopf algebra over $k$. Denote $K$ as the Hopf subalgebra generated by all primitive elements of $H$. Then the following are equivalent:
\begin{itemize}
\item[(1)] $H$ is semisimple.
\item[(2)] $K$ is semisimple.
\item[(3)] $K\cong\left(k N\right)^*$, for $N\cong (C_p)^n$.
\item[(4)] $H\cong\left(k G\right)^*$, for $G$ a $p$-group.
\end{itemize}
\end{thC}

We remark that in \cite{MA}, $H$ is said to be ``irreducible” for ``connected”.  Also, Theorem C is now extended to the generalized super-situation by Corollary 43 of \cite{Masuoka}.

In this short note, we give another proof of Theorem C different from the original one in \cite{MA}. Our approach is discussed in section \ref{idea}, where our aim is to prove Lemma \ref{irreduciblesemisimple}.  Suppose $\Chara (k)=p$ and $H$ is a finite-dimensional connected Hopf algebra over $k$. Denote $K$ as the Hopf subalgebra of $H$ generated by all primitive elements. Let $L$ be a Hopf subalgebra of $H$ such that $K\subseteq L$. Then there exists some $z\in H\setminus L$ such that the comultiplication of $z$ satisfies 
\begin{eqnarray*}
\Delta(z)=z\otimes 1+1\otimes z+u,
\end{eqnarray*}
where the term $u$ represents a none zero cohomological class in $\HL^2(L,k)$, i.e., the Hochschild cohomology of $L$ with coefficients in $k$. Moreover, we say the extension $L\subsetneq H$ is essential if there is no proper Hopf subalgebras between them. Now assume $H$ to be commutative and $L=(kG)^*$ for a $p$-group $G$. Lemma \ref{irreduciblesemisimple} says that in such essential extension $L\subsetneq H$, we can choose a particular $z$ satisfying more nice conditions. As a consequence, $H$ is semisimple. Now we filter $H$ by a sequence of essential extensions starting from $K$:
\begin{eqnarray*}
K=F_0H\subsetneq F_1H\subsetneq \cdots \subsetneq F_{n}H=H.
\end{eqnarray*}
In section \ref{MR}, we will prove $(2)$ implies $(4)$ in Theorem C by induction on the length $n$. The initial step for $n=0$ is just Theorem A due to Hochschild. In the inductive process, we can assume that $H$ is generated by $F_{n-1}H$ and some $z$ with the comultiplication described above. Notice that the main stumbling block is to show $H$ is commutative. The key is to use Lemma \ref{irreduciblesemisimple} for each step $F_{i-1}H\subsetneq F_iH$ and prove that $z$ commutes with $F_{i}$ for all $0\le i\le n-1$ inductively.

Following the main theorem \ref{TMR} and removing the assumption $k=\bar{k}$, we also prove two corollaries. Corollary \ref{c1} says that any semisimple connected Hopf algebra is commutative and Corollary \ref{c2} lists several equivalent conditions for a finite-dimensional connected Hopf algebra to be semisimple. One consequence is that the semisimplicity of such Hopf algebra is completely captured by the semisimplicity of the Hopf subalgebra generated by the first term of its coradical filtration.  

The application of Theorem C can be used to classify all $p^3$-dimensional connected Hopf algebras, which is completed in \cite{VWW,VW}. Suppose $\Chara(k)=p$ and $k=\bar{k}$. Let $H$ be a connected Hopf algebra of dimension $p^3$ over $k$. Assume $H$ is not primitively generated, otherwise it is isomorphic to $u(\mathfrak g)$ for some three-dimensional restricted Lie algebra $\mathfrak g$. Then we can find some Hopf subalgebra $L$ of dimension $p^2$ containing all the primitive elements of $H$. The structure of $L$ is already known, since all $p^2$-dimensional connected Hopf algebra has been classified in \cite{wang2012connected}. Moreover, $H$ can be obtained by adding one non-primitive element to $L$ whose comultiplication is described by $\HL^2(L,k)$. Now according to Theorem $C$, we not only know the structure of $H$ that is semisimple, but also know any nonsemsimple $H$ only arises from nonsemisimple $L$.

\begin{remark}\label{scdim}
Any semisimple Hopf algebra is finite-dimensional by \cite[pp. 107]{sweedlerHopfalgebra}. If furthermore it is connected, then the base field is of positive characteristic. Otherwise the Hopf subalgebra generated by all primitive elements is infinite-dimensional by \cite[Theorem 5.6.5]{montgomery1993hopf}.
\end{remark}


\section{Preliminary}
Throughout this paper, $k$ is a base field. Vector spaces, linear maps and tensor products are assumed to be over $k$ unless stated otherwise. Let $p$ be a prime number, $C_{p^n}$ is the cyclic group of order $p^n$ for a positive integer $n$. By standard notations, we use $(H,m,u,\Delta,\e,S)$ to denote an Hopf algebra. In this section, we recall some basic definitions and facts that will be used throughout the paper.

\begin{deff}\cite[Definitions 5.1.5, 5.2.1]{montgomery1993hopf}\label{cofil}
The \emph{coradical} $H_0$ of $H$ is the sum of all simple subcoalgebras of $H$. The Hopf algebra $H$ is \emph{connected} if $H_0$ is one-dimensional. For each $n\geq 1$, inductively set  \[H_n=\Delta^{-1}(H\otimes H_{n-1} + H_0\otimes H).\] The chain of subcoalgebras $H_0\subseteq H_1 \subseteq \ldots \subseteq H_{n-1}\subseteq H_n \subseteq\ldots $ is the \emph{coradical filtration} of $H$. 
\end{deff}

\begin{deff}\cite[Lemma 1.1]{cstefan1998hochschild}
Consider $k$ as a trivial $H$-bicomodule. We use $\HL^\bullet(k,H)$ to denote the \emph{Hochschild cohomology} of $H$ with coefficients in $k$. It can be computed as the homology of the complex $\left(H^{\otimes \bullet},d^\bullet\right)$, where $d^0=0$ and 
\begin{eqnarray*}
d^n(x)=1\otimes x-(\Delta\otimes \Id_{n-1})(x)+\cdots+(-1)^n(\Id_{n-1}\otimes \Delta)(x)+(-1)^{n+1}x\otimes 1,
\end{eqnarray*}
for $n\ge 1$ and any $x\in H^{\otimes n}$.
\end{deff}

\begin{deff}\cite[V. \S 7 Definition 4]{Jac}
A \emph{restricted Lie algebra} $\mathfrak g$ of characteristic $p$ is a Lie algebra of characteristic $p$ in which there is defined a map $a\to a^{[p]}$ such that
\begin{itemize}
\item[(1)] $(\alpha a)^{[p]}=\alpha^pa^{[p]}$,
\item[(2)] $(a+b)^{[p]}=a^{[p]}+b^{[p]}+\sum_{i=1}^{p-1}s_i(a,b)$, where $is_i(a,b)$ is the coefficient of $\lambda^{i-1}$ in $a\left(\ad(\lambda a+b)\right)^{p-1}$,
\item[(3)] $[ab^{[p]}]=a(\ad b)^p$.
\end{itemize}
for all $a,b\in \mathfrak g$ and $\alpha\in k$. Let $\mathfrak g$ be restricted, and $U(\mathfrak g)$ be the usual enveloping algebra. Suppose $B$ is the ideal in $U(\mathfrak g)$ generated by all $a^p-a^{[p]},x\in \mathfrak g$. Define $u(\mathfrak g)=U(\mathfrak g)/B$, which is called the \emph{restricted enveloping algebra} of $\mathfrak g$. 
\end{deff}

\begin{deff}\cite[Definitions 4.1.1, 7.1.1]{montgomery1993hopf}
A Hopf algebra $H$ \emph{measures} an algebra $A$ if there is a linear map $H\otimes A\to A$, given by $h\otimes a\to h\cdot a$, such that $h\cdot 1=\e(h)1$ and $h\cdot(ab)=\sum(h_1\cdot a)(h_2\cdot b)$, for all $h\in H,a,b\in A$. Moreover, assume that $\sigma$ is an invertible map in $\Hom_k(H\otimes H,A)$. The \emph{crossed product} $A\#_\sigma H$ of $A$ with $H$ is the set $A\otimes H$ as a vector space, with multiplication
\begin{eqnarray*}
(a\#h)(b\#k)=\sum a(h_1\cdot b)\sigma(h_2,k_1)\#h_3k_2,
\end{eqnarray*}
for all $h,k\in H$ and $a,b\in A$.
\end{deff}

\section{Extensions of connected Hopf algebras}\label{idea}
For the remaining of the paper, let $\chara(k)=p$, and $H$ be a finite-dimensional connected Hopf algebra. We use the following notations regarding $H$: 
\begin{itemize}
\item[(1)] The coradical filtration of $H$ is denoted by $\{H_n\}_{n\ge 0}$;
\item[(2)] The Hopf subalgebra of $H$ generated by $H_1$ is denoted by $K$;
\item[(3)] The primitive space of $H$ is denoted by $\mathfrak g$;
\item[(4)] The center of $H$ is denoted by $Z(H)$;
\item[(5)] The augmented ideal of $H$ is denoted by $H^+=\ker \e$;
\item[(6)] Define $\ad(x)y=[x,y]$ and $y(\ad x)=[y,x]$ for any $x,y\in H$.
\end{itemize}

\begin{rem}\label{dim}
Notice that $K\cong u(\mathfrak g)$ and $\dim H=p^n$ for some integer $n$ by \cite[Proposition 2.2(8)]{wang2012connected}.
\end{rem}

\begin{deff}
Let $L$ be a proper Hopf subalebra of $H$. Then $L\subsetneq H$ is called a \emph{weakly essential extension} if there is no proper Hopf subalgebra between them. If moreover $K\subseteq L$, we say it is an \emph{essential extension}.
\end{deff}

\begin{lem}\label{subgen}
Let $L\subsetneq H$ be an extension of  connected Hopf algebras. If there exists some $x\in H\setminus L$ satisfying $\Delta(x)=x\otimes 1+1\otimes x+u$, where $u\in L\otimes L$. Then the subalgebra generated by $L$ and $x$ is a Hopf subalgebra of $H$. Moreover if the extension is weakly essential, then $L$ is equal to $H$.
\end{lem}
\begin{proof}
First of all, the subalgebra generated by $L$ and $x$ is a sub-bialgebra of $H$, which is connected by \cite[Lemma 5.2.12]{montgomery1993hopf}. Therefore it is a Hopf subalgebra because of \cite[Lemma 2.10]{montgomery1993hopf}. Moreover if the extension is weakly essential, then by definition it is equal to $H$.  
\end{proof}

\begin{lem}\label{centerforsemisimpleconnectedHopfalgebras}
If $K$ is semisimple, then $K\subseteq Z(H)$.
\end{lem}
\begin{proof}
By Remark \ref{dim} and Hochschild's result \cite[Theorem 2.3.3]{montgomery1993hopf}, the restricted Lie algebra $\mathfrak g$ is abelian with $\mathfrak g=k\mathfrak g^p$. Fix a basis $\{x_i\,|\,1\le i\le d\}$ for it. Denote $\mathbf{x}=(x_1,x_2,\cdots,x_d)^T$ and $\mathbf{x}^p=(x_1^p,x_2^p,\cdots,x_d^p)^T$. We can write the restricted map for $\mathfrak g$ in a matrix form:
\begin{eqnarray*}
\mathbf{x}^p=R\mathbf{x},
\end{eqnarray*}
for some $R\in \text{GL}_d(k)$. It suffices to prove inductively that $\left[H_n,\mathfrak g\right]=0$ for all $n\ge 0$. Obviously, it is true for $H_0$ and $H_1$. In the following, suppose that $\left[H_n,\mathfrak g\right]=0$ for some $n\ge 1$. By \cite[Lemma 5.3.2(2)]{montgomery1993hopf}, for any $z\in H_{n+1}$, we have $\Delta(z)=z\otimes 1+1\otimes z+u$, where $u\in H_n\otimes H_n$. Therefore by induction
\begin{eqnarray*}
\Delta\left(\left[z,x_i\right]\right)&=&\left[\Delta(z),\Delta(x_i)\right]\\
&=&\left[z\otimes 1+1\otimes z+u,x_i\otimes 1+1\otimes x_i\right]\\
&=&\left[z,x_i\right]\otimes 1+1\otimes \left[z,x_i\right],
\end{eqnarray*}
which yields that $[z,x_i]$ is primitive for all $1\le i\le d$. Therefore we can write these relations in a matrix form as follows:
\begin{eqnarray*}
(\ad z)\mathbf{x}=A\mathbf{x},
\end{eqnarray*}
where $A\in M_d(k)$. Hence a simple calculation shows that:
\begin{eqnarray*}
(\ad z)\mathbf{x}^p=R(\ad z)\mathbf{x}=RA\mathbf{x}.
\end{eqnarray*}
By the definition of restricted Lie algebras, on the other hand, $[z,x_i^p]=z(\ad x_i)^p=0$ for all $1\le i\le d$. Hence $RA=0$ which implies that $A=0$.
\end{proof}

\begin{lem}\label{cocycle}
Let $L\subsetneq H$ be a proper Hopf subalgebra. Then there is a finite set $\{u_i\}\subseteq L\otimes L$, whose image is a basis in $\HL^2(k,L)$. Moreover if $K\subseteq L$, then there exists an element $z\in H\setminus L$ such that $\Delta(z)=z\otimes 1+1\otimes z+\sum\alpha_iu_i$, where the $\alpha_i\in k$ are not all zero.
\end{lem}
\begin{proof}
By the definition of the complex $(L^{\otimes \bullet},d^\bullet)$, $\HL^2(k,L)=\ker d^2/\Img d^1$. Hence there exists a set $\{u_i\}\subseteq \Ker d^2\subseteq L\otimes L$, whose image is a basis in $\HL^2(k,L)$. It is finite since $L$ is finite-dimensional. Furthermore we assume that $K\subseteq L$. As a result of \cite[Theorem 5.2.2(1)]{montgomery1993hopf}, there is a minimal number $d\ (\ge 2)$ such that $L_d\neq H_d$, whence we choose some $z\in H_d\setminus L_d\subseteq H\setminus L$. By \cite[Lemma 5.3.2(2)]{montgomery1993hopf}, $\Delta(z)=z\otimes 1+1\otimes z+u$, where $u\in H_{d-1}\otimes H_{d-1}=L_{d-1}\otimes L_{d-1}\subseteq L\otimes L$. Consider $(L^{\otimes \bullet},d^\bullet)$ as a subcomplex of $(H^{\otimes \bullet},d^\bullet)$. Hence $d^2(u)=d^2d^1(-z)=0$ and $u$ is a $2$-cocycle in the subcomplex.  In regard to the chosen basis $\{u_i\}$ in $\HL^2(k,L)$, there are coefficients $\alpha_i\in k$ and an element $y\in L$ such that 
\begin{eqnarray*}
u-\sum\alpha_iu_i&=&d^1(y)=1\otimes y-\Delta(y)+y\otimes 1;\\
\Delta(z+y)&=&(z+y)\otimes 1+1\otimes (z+y)+\sum\alpha_iu_i.
\end{eqnarray*} 
Finally, $z+y\not\in L$ because of $z\not\in L$ and $y\in L$. The $\alpha_i$ are not all zero, otherwise $z+y\in H_1=K_1\subseteq L$.
\end{proof}

\begin{lem}\label{groupdualcohomologyclass}
Let $H=\left(k G\right)^*$, where $G$ is a $p$-group. Then $H$ is connected and there exists a set $\{u_i\,|\,1\le i\le n\}\subseteq H\otimes H$ satisfying 
\begin{itemize}
\item[(1)] The image of $\{u_i\}$ is a basis in $\HL^2\left(k, H\right)$;
\item[(2)] $u_i^p=u_i$ in $H\otimes H$ for all $1\le i\le n$.
\end{itemize}
\end{lem}
\begin{proof}
Consider $k\supset \mathbb F_p$ as a field extension. It is well known that $\mathbb F_pG$ and $kG$ are scalar local because $G$ is a $p$-group. Therefore by \cite[Proposition 5.2.9(2)]{montgomery1993hopf}, their duals are connected. As shown in \cite[Example 1.3.6]{montgomery1993hopf}, $L:=(\mathbb F_pG)^*$ has a basis $\{f_x|x\in G\}$, where $f_x$ is the characteristic function on the element $x\in G$ and its multiplication and comultiplication structures are described as follows:
\begin{eqnarray*}
f_xf_y&=&\delta_{x,y}f_x,\\
\Delta\left(f_x\right)&=&\sum_{uv=x}f_u\otimes f_v.
\end{eqnarray*}
Obviously, there exists a set $\{u_i|1\le i\le n\}\subseteq L\otimes L$, whose image becomes a basis in $\HL^2(\mathbb F_p,L)$. Moreover each $u_i=\sum \alpha_i f_{x_i}\otimes f_{y_i}$ for some $\alpha_i\in \mathbb F_p$. Hence 
\begin{eqnarray*}
u_i^p&=&\left(\sum \alpha_i f_{x_i}\otimes f_{y_i}\right)^p\\
&=&\sum \alpha_i^pf_{x_i}^p\otimes f_{y_i}^p\\
&=&\sum \alpha_if_{x_i}\otimes f_{y_i}\\
&=&u_i
\end{eqnarray*}
for all  $i$. Furthermore, $\HL^2(\mathbb F_p,L)\otimes k$ is naturally isomorphic to $\HL^2(k,H)$ as $k$-vector spaces because the base change functor $\otimes_{\mathbb F_p}k$ is exact. Then $\{u_i\}$ becomes a basis for the latter through the natural isomorphism.
\end{proof}

\begin{lem}\label{irreduciblesemisimple}
Let $H$ be commutative, and $L\subsetneq H$ be an essential extension. Furthermore assume that $L=\left(k G\right)^*$ for some $p$-group $G$. Then there exists some element $z\in H$ satisfying the following conditions:
\begin{itemize}
\item[(1)] $L$ and $z$ generate $H$ as an algebra.
\item[(2)] $\Delta\left(z\right)=z\otimes 1+1\otimes z+u$, where $u$ is a $2$-cocycle in the complex $\left(L^{\otimes \bullet},d^\bullet\right)$.
\item[(3)] $z^{p^n}\not\in L$ for all $n\ge 0$.
\item[(4)] $z$ satisfies some relation: $z^{p^l}+\lambda_{l-1}z^{p^{l-1}}+\cdots+\lambda_{1}z+a=0$ in $H$, where $\lambda_i\in k$ for all $1\le i\le l-1$ with $\lambda_{1}\neq 0$ and $a\in L$.
\end{itemize}
Moreover, $H$ is semisimple.
\end{lem}
\begin{proof}
Choose a finite set $\{u_i\}\subseteq L\otimes L$, which represents a basis in $\HL^2(k,L)$. By Lemma \ref{groupdualcohomologyclass},  we can assume that $u_i^p=u_i$ for all $i$. Moreover apply Lemma \ref{cocycle} to $L\subsetneq H$, there exists an element $z\in H\setminus L$ such that 
\begin{eqnarray*}
\Delta(z)=z\otimes 1+1\otimes z+\sum\alpha_iu_i, 
\end{eqnarray*}
where the $\alpha_i\in k$ are not all zero. Because of Lemma \ref{subgen}, $H$ is generated by $L$ and $z$ as an algebra. A simple calculation shows that:
\begin{eqnarray*}
\Delta\left(z^{p^n}\right)&=&z^{p^n}\otimes 1+1\otimes z^{p^n}+\left(\sum \alpha_iu_i\right)^{p^n}\\
&=&z^{p^n}\otimes 1+1\otimes z^{p^n}+\sum \alpha_i^{p^n}u_i
\end{eqnarray*}
for all $n\ge 0$. By definition $d^1\left(z^{p^n}\right)=-\sum \alpha_i^{p^n}u_i$, which is never a 2-coboundary in the complex $(L^{\otimes \bullet},d^\bullet)$. Therefore $z^{p^n}\not\in L$ for all $n\ge 0$. Furthermore by \cite[Theorem 4.5]{wang2012connected}, we see $z$ satisfies some relation 
\begin{eqnarray*}
z^{p^l}+\lambda_{l-1}z^{p^{l-1}}+\cdots+\lambda_{1}z+a=0,
\end{eqnarray*}
in $H$, where all $\lambda_i\in k$ with $a\in L$. We want the conditions (1)-(4) being satisfied. But this is true only when removing $\lambda_1\neq 0$ from (4). By (3) we have the smallest index $m$ with $\lambda_m\neq 0$. Re-choosing $z^{p^m}$ for $z$, we see that the conditions are all satisfied. 

Finally, $H$ is isomorphic as an algebra to a certain crossed product $L\#_\sigma (H/L^+H)$ by \cite[Theorem 7.2.11]{dascalescu2000hopf}. The quotient Hopf algebra $H/L^+H$ is generated by the image $\bar{z}$ chosen above satisfying
\begin{eqnarray*}
\bar{z}^{p^l}+\lambda_{l-1}\bar{z}^{p^{l-1}}+\cdots+\lambda_{1}\bar{z}+\e(a)=0,
\end{eqnarray*}
where all $\lambda_i\in k$ with $\lambda_1\neq 0$. Hence it is a polynomial with distinctive roots in $\bar{k}$. Then the quotient Hopf algebra is semisimple and so is $H$ by \cite[Theorem 2.6]{blattner1989crossed}. 
\end{proof}

\section{Main Results}\label{MR}
Recall that $\chara(k)=p$, and $H$ is a finite-dimensional connected Hopf algebra. We use $K$ to denote the Hopf subalgebra generated by the primitive space $\mathfrak g$ of $H$.
\begin{thm}\label{TMR}
Suppose $k$ is algebraically closed. The following are equivalent:
\begin{itemize}
\item[(1)] $H$ is semisimple.
\item[(2)] $K$ is semisimple.
\item[(3)] $K\cong\left(k N\right)^*$, for $N\cong (C_p)^n$.
\item[(4)] $H\cong\left(k G\right)^*$, for $G$ a $p$-group.
\end{itemize}
\end{thm}
\begin{proof}
Condition $(1)$ implies $(2)$ because of \cite[Corollary 2.2.2(2)]{montgomery1993hopf} and Nichols-Zoeller Theorem \cite{nichols1989hopf}. By \cite[Corollary 2.3.5]{montgomery1993hopf}, $(2)$ implies $(3)$. Since group algebras are cosemisimple, their duals are semisimple. Hence $(3)$ implies $(2)$ and $(4)$ implies $(1)$. Because $H$ is finite-dimensional, there exists a finite chain of Hopf subalgebras
\begin{eqnarray*}\label{chain}
K=F_0H\subsetneq F_1H\subsetneq \cdots \subsetneq F_{n}H=H,
\end{eqnarray*}
where each step is an essential extension. We prove $(3)$ implies $(4)$ by induction on the length $n$. When $K=H$, there is nothing to prove. We assume the statement is true for $n\le d$ and suppose that $n=d+1$. By induction when $1\le i\le d$, 
\begin{eqnarray*}
F_iH\cong (kG_i)^*,
\end{eqnarray*}
for $G_i$ a $p$-group. Furthermore if $H$ is commutative, by Lemma \ref{irreduciblesemisimple}, $H$ is semisimple. Hence it is the dual of a group algebra by \cite[Theorem 2.3.1]{montgomery1993hopf}. Remark \ref{dim} gives the order of the group. Thus the inductive step is complete.

It remains to show that $H$ is commutative. We complete it by proving $F_jH\subseteq Z(H)$ inductively again for all $j$. When $j=0$, we have $F_0H=K\subseteq Z(H)$ by Lemma \ref{centerforsemisimpleconnectedHopfalgebras}. Now assume it is true for $j=m$ and let $j=m+1$. In the essential extension $F_dH\subsetneq H$, by Lemma \ref{subgen}, there exists some $z\in H\setminus F_dH$ satisfying: (1) $z$ and $F_dH$ generate $H$; (2) $\Delta\left(z\right)=z\otimes 1+1\otimes z+u$, where $u\in F_dH\otimes F_dH$. Similarly by Lemma \ref{irreduciblesemisimple}, $F_{m+1}H$ is generated by $F_{m}H$ and some $y\in F_{m+1}H\setminus F_{m}H$ such that: (1) $\Delta\left(y\right)=y\otimes 1+1\otimes y+v$, where $v\in F_mH\otimes F_mH$; (2) $y$ satisfies some relation
\begin{eqnarray*}
y^{p^l}+\lambda_{l-1}y^{p^{l-1}}+\cdots+\lambda_{1}y+a=0,
\end{eqnarray*} 
where $\lambda_i\in k$ with $\lambda_{1}\neq 0$ and $a\in F_{m}H$. By induction, $F_{m+1}H\subseteq Z(H)$ if and only if $[F_{m+1}H,z]=0$ if and only if $[y,z]=0$. We wish to show this last $[y,z]=0$. Since $F_m\subseteq Z(H)$ and $F_dH$ is commutative, we have
\begin{eqnarray*}
\Delta\left(\left[y,z\right]\right)&=&\left[\Delta(y),\Delta(z)\right]\\
&=&\left[y\otimes 1+1\otimes y+v,z\otimes 1+1\otimes z+u\right]\\
&=&\left[y,z\right]\otimes 1+1\otimes \left[y,z\right].
\end{eqnarray*}
Hence we can write $[y,z]=x$ for some primitive element $x\in K$ and 
\begin{eqnarray*}
0&=&\left[y^{p^l}+\lambda_{l-1}y^{p^{l-1}}+\cdots+\lambda_{1}y+a,z\right]\\
&=&\left(\ad y\right)^{p^l}(z)+\lambda_{l-1}\left(\ad y\right)^{p^{l-1}}(z)+\cdots+\lambda_1\left[y,z\right]+\left[a,z\right]\\ 
&=&\lambda_1x.
\end{eqnarray*}
Therefore $[y,z]=0$ for $\lambda_1\neq 0$.
\end{proof}
\begin{cor}\label{c1}
If $H$ is semisimple, then $H$ is commutative.
\end{cor}
\begin{proof}
By \cite[Corollary 2.2.2]{montgomery1993hopf}, $H$ is a separable $k$-algebra. Without loss of generality, we can assume $k$ to be algebraically closed of characteristic $p\neq 0$ by a base field extension. Then the result follows from Theorem \ref{TMR}.
\end{proof}

\begin{cor}\label{c2}
The following are equivalent:
\begin{itemize}
\item[(1)] $H$ is semismple.
\item[(2)] $\e(\int^r_H)\neq 0$.
\item[(3)] $\e(\int^l_H)\neq 0$.
\item[(4)] $K$ is semisimple.
\item[(5)] $\mathfrak g$ is abelian and $\mathfrak g=k\mathfrak g^p$.
\item[(6)] $\e(\int^r_K)\neq 0$.
\item[(7)] $\e(\int^l_K)\neq 0$.
\end{itemize}
\end{cor}
\begin{proof}
The equivalence of conditions $(1),(2)$ and $(3)$ is Maschke's Theorem \cite{larson1969associative}. That $(4)$ is equivalent to $(5)$ is Hochschild's result. Let $E\supset k$ be a field extension, and $J$ be the Jacobson radical of $H^*$. We see $J\otimes E$ is the Jacobson radical of $(H\otimes E)^*$ because it is nilpotent and has codimension one. Therefore by \cite[Proposition 5.2.9(2)]{montgomery1993hopf}, $(H\otimes E)_n=H_n\otimes E$ for any $n\ge 0$. By Maschke's theorem \cite{larson1969associative} again, the semisimplicity of Hopf algebras is preserved by base change \cite[Corollary 2.2.2]{montgomery1993hopf}. Therefore we can extend the base filed $k$ to its algebraically closure and apply Theorem \ref{TMR}.
\end{proof}


\begin{thebibliography}{99}
\bibitem{andruskiewitsch2000finite} N.~Andruskiewitsch, About finite dimensional Hopf algebras, \emph{Contemp.~Math.}, 294, Amer. Math. Soc., Providence, RI, (2002), 1--57.
\bibitem{blattner1989crossed} R.J.~Blattner and S.~Montgomery, Crossed products and Galois extensions of Hopf algebras, \emph{Pacific~J.~Math.} 137 (1989), no. 1, 37--54. 
\bibitem{dascalescu2000hopf} S.~D{\u{a}}sc{\u{a}}lescu, C.~N{\u{a}}st{\u{a}}sescu and \c{S}.~Raianu, Hopf algebra: An introduction, vol. 235, Marcel Dekker, New York (2001).
\bibitem{demazure1970groupes} M.~Demazure and P.~Gabriel, \emph{Groupes alg{\'e}briques I}, North Hollad, Amsterdam, 1970.
\bibitem{hochschild1954representations} G.~Hochschild, Representations of restricted Lie algebras of characteristic $p$, \emph{Proc.~Amer.~Math.~Soc.} 5, (1954), 603--605.
\bibitem{Jac}  N.~Jacobson, \emph{Lie Algebras}, Dover Publications Inc., New York, 1979.
\bibitem{larson1988finite} R.G.~Larson and D.E.~Radford, Finite dimensional cosemisimple Hopf algebras in characteristic $0$ are semisimple,  \emph{J.~Algebra} 117(1988), 267-289.
\bibitem{larson1969associative} \bysame, An associative orthogonal bilinear form for Hopf algebras,  \emph{J.~Algebra} 91(1969), 75-93.
\bibitem{Masuoka} A.~Masuoka, Harish-Chandra pairs for algebraic affine supergroup schemes over an arbitrary field, \emph{Transform.~Groups} 17 (2012), no. 4, 1085--1121.
\bibitem{MA} A.~Masuoka, Semisimplicity criteria for irreducible Hopf algebras in positive characteristic, \emph{Proc.~Amer.~Math.~Soc.} 137 (2009), no. 6, 1925--1932.
\bibitem{montgomery1993hopf}  S.~Montgomery, \emph{Hopf Algebras and Their Actions on Rings},  Amer.~Math.~Soc., 82(1993).
\bibitem{Na} M.~Nagata, Complete reducibility of rational representations of a matric group, \emph{J.~Math.~Kyoto~Univ.} 1 (1961/1962), 87--99.
 \bibitem{VWW} V.~Nguyen, L.~Wang and X.~Wang, Classification of connected Hopf algebras of dimension $p^3$ I, \emph{J. Algebra} 424 (2015), 473--505.
 \bibitem{VW} V.~Nguyen, L.~Wang and X.~Wang, Primitive deformations of quantum $p$-groups, to appear \emph{Algebr.~Represent.~Theory.}
\bibitem{nichols1989hopf} W.D.~Nichols and M.B.~Zoeller, A Hopf algebra freeness theorem, \emph{Amer.~J.~Math.} 111 (1989), no. 2, 381--385.
\bibitem{cstefan1998hochschild}D.~\c{S}tefan and F.V.~Oystaeyen, Hochschild cohomology and the coradical filtration of pointed coalgebras: applications, \emph{J. Algebra} 210 (1998), no. 2, 535--556. 
\bibitem{sweedler1971connected}M.E.~Sweedler, Connected fully reducible affine group schemes in positive characteristic are abelian, \emph{J.~Math.~Kyoto~Univ.} 11 (1971), 51--70. 
\bibitem{sweedlerHopfalgebra}\bysame, \emph{Hopf algebras}, Benjamin, New York, 1969.
\bibitem{wang2012connected} X.~Wang, Connected Hopf algebras of dimension $p^2$, \emph{J.~Algebra}, 391 (2013), 93--113.
\end{thebibliography}
\end{document}